\documentclass{amsart}
\usepackage{amsmath,amssymb, amscd}
\usepackage{amsfonts}
\usepackage{xcolor}

\newtheorem{prop}{{\bf Proposition}}[section]
\newtheorem{coro}[prop]{{\bf Corollary}}
\newtheorem{lemma}[prop]{{\bf Lemma}} 
\newtheorem{theor}[prop]{{\bf Theorem}} 
\newtheorem{ex}[prop]{{\bf Example}} 

\def\F{{\mathbb F}}

\def\ad{\textup{ad}}

\def\chn{\color{black}}

\begin{document}
\author{Pilar P\'{a}ez-Guill\'{a}n}
\address{Departamento de Matem\'{a}ticas, Universidade de Santiago de Compostela,
R\'{u}a Lope G\'{o}mez de Marzoa, 15782--Santiago de Compostela, Spain}
\email{pilar.paez@usc.es}
\author{Salvatore Siciliano}
\address{Dipartimento di Matematica e Fisica ``Ennio De Giorgi", Universit\`{a} del Salento,
Via Provinciale Lecce--Arnesano, 73100--Lecce, Italy}
\email{salvatore.siciliano@unisalento.it}
\author{David A. Towers}
\address{Lancaster University\\
Department of Mathematics and Statistics \\
LA$1$ $4$YF Lancaster\\
England}
\email{d.towers@lancaster.ac.uk}


\begin{abstract} In this paper we study the lattice of restricted subalgebras of a restricted Lie algebra. In particular, we consider those algebras in which this lattice is dually atomistic, lower or upper semimodular, or in which every restricted subalgebra is a quasi-ideal. The fact that there are one-dimensional subalgebras which are not restricted results in some of these conditions being weaker than for the corresponding conditions in the non-restricted case.
\end{abstract}
\title[On the subalgebra lattice of a restricted Lie algebra]{On the subalgebra lattice of a restricted Lie algebra}
\maketitle
\section{Introduction}
The relationship between the structure of a group and that of its lattice of subgroups is highly developed and has attracted the attention of many leading algebraists. According to Schmidt (\cite{schmidt}), the origin of the subject can be traced back to Dedekind, who studied the lattice of ideals in a ring of algebraic integers; he discovered and used the modular identity, which is also called the Dedekind law, in his calculation of ideals. Since then modularity and lattice conditions related to it have been studied in a number of contexts. The lattice of submodules of a module over a ring is modular, and hence so is the lattice of subgroups of an abelian group. The lattice of normal subgroups of a group is also modular, but the lattice of all subgroups is not in general. 
\par

The study of the subalgebra lattice of a finite-dimensional Lie algebra was popular in the 1980's and in the 90's (see, for example, \cite{1,7,gein,3,gein2,10,kol,la,13,15,5,12,6,9}), but interest then waned.
The likely reason is that most of the conditions under investigation were too strong and so few algebras satisfied them. However, the lattice of restricted subalgebras of a restricted Lie algebra is fundamentally different; for example, not every element spans a one-dimensional restricted subalgebra. Thus, one could expect more interesting results to hold for restricted algebras and, as we shall see, this is indeed the case.
\par

In Section 2 we fix some notation and terminology and introduce some results that are needed later. In Section 3 we study restricted Lie algebras that are dually atomistic; that is, such that every restricted subalgebra is an intersection of maximal restricted subalgebras. We show that such algebras over an algebraically closed  field of positive characteristic are solvable or semisimple, and then characterise the solvable ones more precisely. It turns out that they are more abundant than in the non-restricted case. We then investigate those restricted Lie algebras all of whose subalgebras (not necessarily restricted) are intersections of maximals. It is shown  that if the ground field is algebraically closed then there are no such algebras that are perfect.
\par

The objective in Section 4 is to study restricted Lie algebras $L$ in which every restricted subalgebra is a restricted quasi-ideal; that is, such that $[S,H] \subseteq S+H$ for all restricted subalgebras $S,H$ of $L$. These are characterised over an algebraically closed field of characteristic different from 2, and the nilpotent ones more generally over a perfect field of characteristic different from 2. In this regard, we also mention that restricted Lie algebras over perfect fields all of whose restricted subalgebras are ideals were characterized by the second author in \cite{SS}. Section 5 then goes on to consider $J$-algebras and lower semimodular restricted Lie algebras. The main result here is the same as for the corresponding situation as in the non-restricted case if the underlying field is algebraically closed, but it is pointed out that the assumption of algebraic closure cannot be omitted.
\par

The final section is devoted to studying upper semimodular restricted Lie algebras.  It is shown that, over an algebraically closed field, such an algebra is either almost abelian or nilpotent of class at most $2$. This is proved by considering first upper semimodular restricted Lie algebras that are generated by their one-dimensional restricted subalgebras.  It is also shown that over such fields the conditions that $L$ is modular, $L$ is upper semimodular and every restricted subalgebra of $L$ is a quasi-ideal are equivalent.



\section{Preliminaries}
Here we fix some notation and terminology and introduce some results that will be needed later.  Unless otherwise stated, throughout the paper all algebras are supposed finite-dimensional. Let $L$ be a Lie algebra over a field $\F$.
As usual, the {\em derived series} for $L$ is defined inductively by $L^{(0)}=L$, $L^{(k+1)}=[L^{(k)},L^{(k)}]$ for $k\geq 0$, $L^{(\infty)}=\cap_{k\geq 0}L^{(k)}$; $L$ is {\em solvable} if $L^{(\infty)}=0$. Similarly, the {\em lower central series} is defined inductively by $L^1=L$, $L^{k+1}=[L^k,L]$ for $k\geq 1$; $L$ is {\em nilpotent} if $L^k=0$ for some $k\geq 1$.  Also, $L$ is said to be \emph{supersolvable} if it admits a complete flag made up of ideals of $L$, that is, there exists a chain 
$$
0=L_0\subsetneq L_1 \subsetneq \cdots \subsetneq L_n=L
$$
of ideals of $L$ such that $\dim L_j=j$ for every $0\leq j\leq n$.  The centre of $L$ is denoted by $Z(L)$,  and $C_B(A)=\{x\in B\colon [x,A]=0\}$ denotes the {\it centraliser} in a subalgebra $B$ of another subalgebra $A$. Also, the {\it ascending central series} is defined inductively by $C_1(L)=Z(L)$, $C_{n+1}(L)=\{x\in L\colon [x,L]\subseteq C_n(L)\}$. The {\em nilradical} $N(L)$ is defined to be the maximal nilpotent ideal, and the {\it solvable radical}, denoted by $R(L)$, is defined to be the maximal solvable ideal. For every $x\in L$, the adjoint map of $x$ is defined by $\ad(x)\colon L \rightarrow L$, $a\mapsto [x,a]$. If $S$ is a subalgebra of $L$, then the largest ideal of $L$ contained in $S$ is called the {\em core} of $L$ and is denoted by $S_L$. The {\em Frattini subalgebra} $F(L)$ of $L$ is the intersection of all maximal subalgebras of $L$; the {\em Frattini ideal} of $L$ is $\phi(L)=F(L)_L$. The {\em abelian socle}, $Asoc(L)$, is the sum of the minimal abelian ideals of $L$. We will denote algebra direct sums by $\oplus$, whereas direct sums of the vector space structure alone will be written as $\dot{+}$.

We say that $L$ is {\em dually atomistic} if every subalgebra of $L$ is an intersection of maximal subalgebras of $L$. The Lie algebra $L$ is said to be {\em almost abelian} if $L=\F x\dot{+} A$, where $A$ is an abelian 
ideal of $L$ and $\ad (x)$  acts  as the identity map on $A$.
Scheiderer proved in \cite{sch} that, over a field of characteristic zero, every dually atomistic Lie algebra is abelian, almost abelian or simple. Here we establish a slightly weaker version of this which is valid over any field. 

\begin{prop}\label{da} If $L$ is a dually atomistic Lie algebra over any field then $L$ is either
  abelian, almost abelian or semisimple.
\end{prop}

\begin{proof} Let $L$ be dually atomistic and suppose that $L$ is not
  semisimple.  Then $Asoc (L) \neq 0$ and $L$ splits over $Asoc( L)$, by
 \cite[Theorem 7.3]{frat}.  Furthermore, the minimal abelian ideals of
  $L$ are one-dimensional, by  \cite[Lemma 1]{sch}, so we can write $L =
  (\F a_{1} \oplus \dots \oplus \F a_{n}) \dot{+} B$, where $\F a_{i}$ is a
  minimal ideal of $L$ for each $1 \leq i \leq n$, $B$ is a subalgebra
  of $L$, and $n \geq 1$.
  \par
  Let $M$ be a maximal subalgebra of $L$ with $a_{1} \not \in M$. We
  shall show that $L^{(\infty)} \subseteq M$. Now $L = \F a_{1} + M$, so
  $M$ has codimension one in $L$.  It follows that $L/M_{L}$ is as
  described in \cite[Theorems 3.1 and 3.2]{ama}.  Also,
  $[\F a_{1},M_{L}] \subseteq \F a_{1} \cap M = 0$.  We consider the three
  cases given in \cite[Theorem 3.1]{ama} separately. 
  \par
  \emph{Case (a)}: Here $L/M_{L}$ is one-dimensional, so $M =
  M_{L}$ and $L^{2} \subseteq M$.
  \par
  \emph{Case (b)}: Here $L/M_{L}$ is two-dimensional, so $L =
  \F a_{1} + \F m + M_{L}$ where $m \in M \setminus M_{L}$.  Now $L^{2}
  \subseteq \F a_{1} + M_{L}$ and $L^{(2)} \subseteq M_{L} \subseteq
  M$. 
  \par
  \emph{Case (c)}: Here $L/M_{L} \simeq L_{m}(\Gamma)$. If $m$ is
  odd then $L_{m}(\Gamma)$ is simple.  But $(\F a_{1} + M_{L})/M_{L}$ is
  a one-dimensional ideal of $L/M_{L}$, which is a contradiction.  Hence
  $m$ is even, in which case $L_{m}(\Gamma) = \F x + L_{m}(\Gamma)^{2}$, where
  $L_{m}(\Gamma)^{2}$  is simple.  Put $L/M_{L} = \bar{L}$, and so on.  Then
  $\bar{L} = \F  \bar{a_{1}} \oplus \bar{L}^{2}$ and $ [\bar{L},
  \bar{a_{1}}]  = \bar{0}$; that is, $ [L,a_{1}]  \subseteq M_{L}$, whence
  $L^{2} \subseteq M$.
  \par
  In any case we have established that, for any maximal subalgebra $M$
  of $L$, either $a_{1} \in M$ or $L^{(\infty)} \subseteq M$.  Suppose
  that $L^{(\infty)} \neq 0$.  Then $L^{(\infty)} \neq \F a_{1}$ (since
  $(\F a_{1})^{2} = 0$), so there is an element $x \in L^{(\infty)}
  \setminus \F a_{1}$.  Let $M$ be a maximal subalgebra containing $x +
  a_{1}$.  Then either $a_{1} \in M$ or $L^{(\infty)} \subseteq M$.
  In each case, $\F x + \F a_{1} \subseteq M$.  It follows that $\F (x +
  a_{1})$ cannot be an intersection of maximal subalgebras of $L$, a
  contradiction.  Hence, $L^{(\infty)} = 0$ and $L$ is solvable. The
  result now follows from \cite[Lemma 1]{sch}. 
\end{proof}

We shall need the following result which is due to Grunewald, Kunyavskii, Nikolova and Plotkin for $p>5$. However, the same proof works for $p>3$ by using the Corollary in page 180 of~\cite{skr2}. 
A. Premet has pointed out that the result is also valid for $p=3$, but that it relies on results that have not yet been published, so we omit this case.

\begin{lemma}\label{gknp} Every simple Lie algebra $L$ over an algebraically closed field $\F $ of characteristic $p>3$ contains a subalgebra $S$ with a  quotient isomorphic to $\mathfrak{sl}(2,\F )$.
\end{lemma}
\begin{proof} The proof is the same as for \cite[Lemma 3.2]{gknp} with the reference to \cite[Part II,  Corollary 1.4]{weis} replaced  by~\cite[page 180, Corollary]{skr2}. 
\end{proof}


In what follows we shall be studying the lattice of restricted subalgebras of a restricted Lie algebra. 
Let $L$ be a restricted Lie algebra over a field of characteristic $p>0$. For a subset $S$ of $L$, we denote by $\langle S\rangle_p$ the restricted subalgebra generated by $S$. We say that $L$ is \emph{cyclic} if $L=\langle x \rangle_p$ for some $x\in L$. An element $x\in L$ is said to be \emph{semisimple} if $x\in \langle x^{[p]}\rangle_p$ and \emph{toral} if $x^{[p]}=x$. An abelian restricted Lie algebra consisting of semisimple elements is called a \emph{torus}. An element $x\in L$ is said to be \emph{$p$-nilpotent} if $x^{[p]^n}=0$ for some $n>0$ (in this case, the minimal $n$ with such a property is called the \emph{order} of $x$), and $L$ is said to be \emph{$p$-nilpotent} if there exists $n>0$ such that $x^{[p]^n}=0$ for every $x\in L$.  We also introduce as ``restricted analogues" of earlier concepts, $F_p (L)$, the {\em Frattini $p$-subalgebra} of $L$, to be the intersection of the maximal restricted subalgebras of $L$, and $\phi_p(L)$, the {\em Frattini
$p$-ideal} of $L$, to be the largest restricted ideal of $L$ which is contained in $F_p(L)$. The {\em abelian $p$-socle}, $A_psoc(L)$, is the sum of the minimal abelian restricted ideals of $L$.  To avoid tedious repetition we shall therefore often omit the word 'restricted'.


%
%
%
%
%

\section{Dually atomistic Lie algebras}\label{s:da}
We say that a restricted Lie algebra $L$ is {\em dually atomistic} if every restricted subalgebra of $L$ is an intersection of maximal restricted subalgebras of $L$. It is easy to see that if $L$ is dually atomistic then so is every factor algebra of $L$, and if $L$ is dually atomistic then it is $\phi_p$-free.  

\begin{lemma}\label{l:nilrad} Let $L$ be a dually atomistic restricted Lie algebra.   Then:
	\begin{enumerate}
		\item[(i)] $N(L)$ is abelian;
		\item[(ii)] $M \cap N(L)$ is a restricted ideal of $L$ for every maximal restricted subalgebra
		$M$ of $L$; and
		\item[(iii)] for every subspace $S$ of $N(L)$, $\langle S\rangle_p$ is a restricted ideal of $L$. 
	\end{enumerate}
\end{lemma}

\begin{proof} (i) $N(L)^2\subseteq \phi_p(L) = 0$ by~\cite[Theorem 6.5]{frat} and~\cite[Theorem 3.5]{lt1}.
	
	(ii) The result is clear if $N(L) \subseteq M$, so suppose that
	$N(L) \not \subseteq M$.  Then $L = N(L) + M$ and 
	\begin{align*}
		[L,N(L) \cap M]  &= [N(L) + M,N(L) \cap M]\\
		&\subseteq  N(L)^{2} + N(L)\cap M^{2} 
		\subseteq  N(L) \cap M,
	\end{align*}
	using (i).
	
	(iii) By (i), every subspace of $N(L)$
	is a subalgebra of $L$.  Let $S$ be any subspace of $N(L)$.  Then
	$$
	\langle S\rangle_p  = \langle S\rangle_p \cap N(L)   = \left(\bigcap_{M \in \mathcal{M}} M\right) \cap N(L)  = \bigcap_{M \in \mathcal{M}} (M \cap N(L)),
	$$
	where $\mathcal{M}$ is the set consisting of all maximal restricted subalgebras of $L$
	containing $\langle S\rangle_p$.  Therefore, $\langle S\rangle_p$ is an intersection of restricted ideals of $L$,
	by (ii), and so is itself a restricted ideal of $L$.
\end{proof}
\chn

\begin{prop} Let $L$ be a dually atomistic restricted Lie algebra over an algebraically closed field $\F $. Then $L$ is solvable or semisimple.
\end{prop}
\begin{proof} Suppose that $L$ is not semisimple. Then $L=N(L)\dot{+}B= A_1\oplus \cdots \oplus A_n\dot{+}B$, where $B$ is a restricted subalgebra of $L$ and $A_1\oplus \cdots \oplus A_n=A_psoc(L)\neq 0$, by \cite[Theorems  3.4  and 4.2]{lt1}.  If $B=0$, then $L$ is nilpotent and we are done. Assume therefore that $B\neq 0$. 
	\par
	
	Suppose first that $N(L)=Z(L)$. Then,   $L=Z(L)\oplus B$ and $L^2\subseteq B$.  Then we must have that $N(L)=R(L)$. 
	For, otherwise, there is a minimal ideal $A/N(L)$ of $L/N(L)$ with $A\subseteq R(L)$. But $A$ is nilpotent, which is a contradiction. Thus, $B$ is semisimple  and $Z(L)\neq 0$.  Let $M$ be a maximal restricted subalgebra of $L$. If $Z(L)$  is not contained in $M$ then $M+Z(L)$ is a restricted subalgebra properly containing $M$, so $L = M + Z(L)$ and $\langle B^2\rangle_p = \langle L^2\rangle_p \subseteq M$, since $L^2\subseteq M$ and $M$ is restricted. Hence, either $Z(L)$ or $\langle B^2\rangle_p$ is inside $M$. 
	\par
	
	Let $z \in Z(L)$ and $b \in \langle B^2\rangle_p$, and let $M$ be a maximal restricted subalgebra containing $\langle z + b\rangle_p$. Then $z, b \in M$, so we must have $\langle z\rangle_p + \langle b\rangle_p = \langle z + b\rangle_p$. But then $b = \sum_{i=0}^n \lambda_i (b^{[p]^i} + z^{[p]^i})$, so $b =\sum_{i=0}^n \lambda_i b^{[p]^i}$   and $\sum_{i=0}^n \lambda_i z^{[p]^i} = 0$. If $b$ is not semisimple, then $\lambda_0 = 1$ which implies that $z$ is semisimple, from the second sum. This must hold for every choice of $z \in Z(L)$, so $Z(L)$ is a torus of $L$,  by~\cite[Chapter 2, Theorem 3.10]{sf}. A similar argument shows that if $z$ is not semisimple then every $b$ must be, in which case $\langle B^2\rangle_p$ is a torus of $L$. Hence, either $Z(L)$ or $\langle B^2\rangle_p$ is a torus. In the latter case, $\langle B^2\rangle_p$ is abelian, contradicting the fact that $B$ is semisimple. In the former case, both $Z(L)$ and $\langle B^2\rangle_p$ have a toral element: $z$ and $b$, say. But then $\langle z\rangle_p+\langle b\rangle_p=\F z+\F b\neq \F (z+b)=\langle z+b\rangle_p$, a contradiction. 
	\par
	
	Therefore suppose that $N(L) \neq Z(L)$. Then there is a minimal restricted ideal $A$ with $A\subseteq N(L)$ and $A\cap Z(L)=0$. Moreover, if $a\in A$, we have that $a^{[p]}\in A\cap Z(L)$, so $A=\F a$ with $a^{[p]} =0$, by Lemma \ref{l:nilrad}(iii). Let $M$ be a maximal restricted subalgebra of $L$  such that $a\notin M$.  We have $L=M\dot{+}A$, by \cite[Lemma 2.1]{lt1}, so $M$ has codimension one in $L$, and, as in Proposition \ref{da}, $\langle L^{(\infty)}\rangle_p \subseteq M$. It follows that $\langle L^{(\infty)}\rangle_p\cap A=0$. Choose 
	$x \in \langle L^{(\infty)}\rangle_p$. Then $[x,a]\in L^{(\infty)}\cap A=0$. If $\langle x+a\rangle_p = \langle x\rangle_p+\langle a\rangle_p$, then we have $a = \sum_{i=0}^n\lambda_i(x + a)^{[p]^i}=\lambda_0a + \sum_{i=0}^n\lambda_ix^{[p]^i}$. Hence $\lambda_0=1$ and $x$ is semisimple. It follows  from~\cite[Chapter 2, Theorem 3.10]{sf}  that $\langle L^{(\infty)}\rangle_p$ is abelian. But this means that $L$ is solvable.
\end{proof}

For a field $\F $ of characteristic $p>0$, we will denote by $\F [t,\sigma]$ the skew polynomial ring over $\F $ in the indeterminate $t$ with respect to the Frobenius endomorphism $\sigma$ of $\F $. We recall  $\F [t,\sigma]$ is the ring consisting of all polynomials $f=\sum_{i\geq0}\alpha_i t^{i}$ with respect to the usual sum and multiplication defined by the condition $t\cdot\alpha=\alpha^p t$ for every $\alpha \in \F $.

\begin{prop}\label{p:solvable}  Let $L$ be a solvable restricted Lie algebra over
 any field $\F $.  If $L$ is dually atomistic then 
\[L\simeq (\mathcal{L}/\langle \bar{f_1}\rangle_p \oplus \cdots \mathcal{L}/\langle \bar{f_r}\rangle_p \oplus \F x_{r+1} \oplus \cdots \oplus \F x_n)\dot{+}\F b, \]
 where $r\geq 0$, but $r\neq n$, $b$ is toral, $\mathcal{L}=\langle x\rangle_p$ is a free cyclic restricted Lie algebra and $\bar{f_i} = \sum_{k= 0}^{s} \alpha_k x^{{[p]}k}$ is an element of $\mathcal{L}$ such that
$f_i= \sum_{k= 0}^s \alpha_k t^k$ is an irreducible element of the ring $\F [t,\sigma]$.
\end{prop}

\begin{proof} 
The nilradical $N(L)$ of $L$ is  non-zero and  abelian by Lemma~\ref{l:nilrad}(i). As $L$ is $\phi_p$-free, $L=N(L)\dot{+} B$ for some restricted subalgebra $B$ of $L$, and $N(L) = A_psoc(L):= A$, by \cite[Theorems 3.4 and 4.2]{lt1}. Let $a\in A$. Then  $C_B(A)$  is a restricted ideal of $L$ and $C_B(A)\cap  A  =0$, so $C_B(A)=0$ and $B$ acts faithfully on $A$. Also $\ad^2(a)=0$ and so $\ad(a^{[p]})=0$, whence $a^{[p]}\in Z(L)$ for all $a\in A$.
\par

We can write $A=A_1\oplus \dots \oplus A_n$, where $A_i$ is a minimal abelian restricted ideal of $L$ for $1\leq i\leq n$. Moreover, $A_i \simeq \mathcal{L}/\langle \bar{f_i}\rangle_p$, where $\bar{f_i} = \sum_{k\geq 0} \alpha_k x^{{[p]}k}$ is an element of $\mathcal{L}$ such that
$f_i= \sum_{k\geq 0} \alpha_k t^k$ is an irreducible element of the ring $\F [t,\sigma]$, by Lemma \ref{l:nilrad}(iii) and \cite[Proposition 3.1]{mps}. Let $A_1\oplus \dots \oplus A_r=Z(L)$, where we allow that $r$ could be $0$. Since $B$ acts faithfully on $A$ we cannot have $r=n$. Then $[B,A]=A_{r+1}\oplus \dots \oplus A_n=\F x_{r+1}\oplus \dots \oplus \F x_n$. Now $C_B(x_i)$ is a restricted ideal of $L$, so $C_B(x_i)=0$ for each $r+1\leq i\leq n$. Let $b_1,b_2\in B$. Then $[b_i,x_n]=\lambda_i x_n$ for some $0\neq \lambda_i\in \F $, $i=1,2$. But then $[\lambda_2b_1-\lambda_1b_2,x_n]=0$, whence $b_1$ and $b_2$ are linearly dependent and $B$ is one-dimensional. Choose $B=\F b$ such that $[b,x_n]=x_n$.  Let $b^{[p]}=\mu b$. Then
\[ x_n=[b^{[p]},x_n]=\mu[b,x_n]=\mu x_n,
\] so $\mu=1$ and $b$ is toral.
\end{proof}




We introduce another piece of notation before presenting the following results. We say that a Lie algebra is {\it restricted dually atomistic} if it is restricted and every subalgebra is an intersection of maximal subalgebras.



\begin{prop}\label{p:perfect}
Let $L$ be a perfect restricted dually atomistic Lie algebra over any field $\F $ of characteristic $p>0$. Then every subalgebra of $L$ is restricted.
\end{prop}
\begin{proof} Arguing as in \cite[Lemma 3.7]{lt1}, it is immediate to prove that every maximal subalgebra of $L$ is self-idealising. It follows from \cite[Lemma 3.9]{lt1} that every maximal subalgebra of $L$ is restricted. The result now follows from the fact that $L$ is dually atomistic.   
\end{proof}


\begin{theor} There are no perfect restricted dually atomistic Lie algebras over an algebraically closed field.
\end{theor}
\begin{proof} Suppose that $L$ is a counterexample of minimal dimension. By Proposition~\ref{p:perfect}, $L$ is simple as a Lie algebra, and hence its absolute toral rank is just the dimension of a maximal torus $T$ (cf. \cite[\S 1.2]{strade}). Given two linearly independent elements $x,y\in T$, Proposition~\ref{p:perfect} forces $0\neq (x+\lambda y)^{[p]}\in \F  (x+\lambda y)$  for all $\lambda\in \F $, but this cannot happen since $\F $ is algebraically closed. Hence, $L$ has absolute toral rank $1$.
	
 Now, if $\F $ has characteristic $p=2,3$, then \cite[Theorem 6.5]{skr} yields that $L$ is solvable or isomorphic to $\mathfrak{sl}(2,\F )$ or to $\mathfrak{psl}(3,\F )$. Otherwise, $L$ has a restricted subalgebra with a quotient isomorphic to $\mathfrak{sl}(2,\F )$, by Lemma \ref{gknp} and Proposition~\ref{p:perfect}.  But both $\mathfrak{sl}(2,\F )$ and $\mathfrak{psl}(3,\F )$ have elements which are neither semisimple nor $p$-nilpotent, which clearly contradicts Proposition~\ref{p:perfect}.
\end{proof}


As well as the three-dimensional non-split simple Lie algebra, which is dually atomistic in the characteristic zero case, there exist other perfect dually atomistic simple restricted Lie algebras over a perfect field which is not algebraically closed. For example, let $L$ be the seven-dimensional simple Lie algebra over a perfect field of characteristic 3 constructed by Gein in \cite[Example 2]{gein}. 
This algebra $L$ can be endowed with a $[p]$-mapping such that every element is semisimple. Any two linearly independent elements of $L$ generate a three-dimensional non-split restricted subalgebra which is maximal in $L$. Any second-maximal restricted subalgebra is then one-dimensional, and every one-dimensional restricted subalgebra $S$ is inside more than one maximal restricted subalgebra whose intersection is $S$.

We finish this section by studying the so-called {\it atomistic} restricted Lie algebras, those in which every restricted subalgebra is generated by minimal restricted subalgebras.

\begin{prop}\label{prop:atom}
	Let $\F $ be an algebraically closed field of characteristic $p>0$. A restricted Lie algebra $L$ over $\F $ is atomistic if and only if every $p$-nilpotent cyclic restricted subalgebra is one-dimensional.
\end{prop}

\begin{proof}
	Note that $L$ is atomistic if and only if all its cyclic restricted subalgebras are atomistic. Consider the cyclic restricted subalgebra $C$, whose semisimple elements form a torus $T$, and whose $p$-nilpotent elements form a $p$-nilpotent restricted subalgebra $P$. By~\cite[Chapter 2, Theorem 3.6]{sf}, $T$ is atomistic. From~\cite[Chapter 2, Theorem 3.5]{sf}, it follows that $C=T\oplus P$, so 
	$C$ is atomistic if and only if $P$ is atomistic. But this is equivalent to requiring that $\dim P =1$ (\cite[Theorem 3.8]{mps}). 
	The result follows.
\end{proof}

\section{Restricted quasi-ideals }
 A restricted subalgebra $S$ of $L$ is called a {\em restricted quasi-ideal} of $L$ if $[S,H]\subseteq S+H$ for all restricted subalgebras $H$ of $L$. Clearly, every restricted subalgebra that is a quasi-ideal is also a restricted quasi-ideal.
\par

Denote by $L^{[p]}$ the restricted subalgebra generated by all the elements $x^{[p]}$, with $x\in L$.

\begin{lemma} If $S$ is a restricted subalgebra of $L$, then $S_L$ is a restricted ideal of $L$
\end{lemma}
\begin{proof} Simply note that $(S_L)^{[p]}$ is an ideal of $L$ inside $S$.
\end{proof}
\begin{prop}\label{p:qii} If $\F $ is perfect then $L^{[p]}$ is a restricted quasi-ideal if and only if it is an ideal of $L$.
\end{prop} 
\begin{proof} Suppose that $L^{[p]}$ is a restricted quasi-ideal of $L$. Then, for all $x\in L$
\[ [L^{[p]},x]\subseteq L^{[p]}+\langle x\rangle _p=L^{[p]}+\F x,
\] so $L^{[p]}$ is a quasi-ideal.  Suppose that $L^{[p]}$ is not an ideal of $L$,  and factor out $(L^{[p]})_L$, so we can assume that $L^{[p]}$ is core-free.  Then, by  \cite[Theorem 3.6]{ama1},  there are three possibilities which we will consider in turn.
\par

Suppose first that $L^{[p]}$ has codimension 1 in $L$. Define $(L^{[p]})_i$ as in \cite[(5)]{ama}. Then every element $x\in L$ can be written as $x=x_s+x_n$, where $x_s$ is semisimple and $x_n$ is $p$-nilpotent, by \cite[Theorem 3.5]{sf}. Moreover, all semisimple elements belong to $L^{[p]}$, so $L=L^{[p]}+\F x$ for some $p$-nilpotent element $x$. Suppose that $x^{[p]^k}=0$. Now $(L^{[p]})_i= \{y\in L^{[p]} \mid [y,_i x]\in L^{[p]}\}$ for $i\geq 0$, by \cite[Lemma 2.1(b)]{ama}. Hence $[y,_{p^h}x]=[y,x^{[p]^h}]=0$ for $h \geq k$. Also, $(L^{[p]})_0=L^{[p]}$ and $(L^{[p]})_{i+1}\subseteq (L^{[p]})_i$ for $i\geq 0$, so $(L^{[p]})_L=\cap_{i=0}^{\infty} (L^{[p]})_i=L^{[p]}$, by \cite[Lemma 2.1]{ama}, contradicting the fact that $L^{[p]}$ is not an ideal of $L$.
\par

On the other hand, \cite[Theorem 3.6(c)]{ama1} cannot hold, as the three-dimensional simple Lie algebra $W(1,2)^2$ over a field of characteristic $2$ is not restrictable. To see this simply note that the derivation $\ad^2(x)$ is not inner.
\par

Finally, suppose that \cite[Theorem 3.6(d)]{ama1} holds. Then $L=L^2+\F y$ where $\ad(y)$ acts as the identity map on $L^2$ and $L^{[p]}=\F y$.  Let $x\in L^2$.  We have $\ad^p(y)=\ad(y)$ and $\ad^p(x)=0$ for every $x\in L^2$. Therefore, as $L$ is centerless, the $p$-mapping of $L$ is determined by the conditions $y^{[p]}=y$ and $x^{[p]}=0$. This implies $L^{[p]}=L$, a contradiction.
\par

The converse is straightforward.
\end{proof}

\begin{prop}\label{p:car2} Let $L$ be a restricted Lie algebra such that every restricted subalgebra of $L$ is a restricted quasi-ideal. Then $L^2\subseteq L^{[p]}$. It follows that $L^3=L^{p+1}$; in particular, if  $L$ is nilpotent, then $L$ has nilpotency class at most $2$. 
\end{prop} 
\begin{proof} By Proposition~\ref{p:qii}, $L^{[p]}$ is a restricted ideal.  Put $\mathfrak{L}=L/L^{[p]}$. Then  $\mathfrak{L}^{[p]}=0$ and every subalgebra of $\mathfrak{L}$ is a quasi-ideal. If $\mathfrak{L}$ is not abelian then it is almost abelian, by \cite[Theorem 3.8]{ama1}, so $\mathfrak{L}=\mathfrak{L}^2+\F y$, where $\ad(y)$ acts as the identity map on $\mathfrak{L}^2$. But then, if $0\neq x\in \mathfrak{L}^2$, then $0=[y^{[p]},x]=x$, a contradiction. It follows that $\mathfrak{L}^2=0$, so $L^2\subseteq L^{[p]}$. Now, if 
	$p\neq 2$, then we are done. Assume then that $p=2$, and suppose, by contradiction, that $L$ has nilpotency class $n>2$. Set $H=L/C_{n-3}(L)$, which has nilpotency class $3$. By~\cite[Chapter 16, Proposition 1.1]{amast}, $H$ does not satisfy the second Engel condition, and therefore there exist $x,y\in H$ such that $[x,y^{[2]}]=[[x,y],y]\neq 0$.  Set $\tilde{x},\tilde{y}$ to be preimages of $x,y$ in $L$, and note that $\tilde{x}^{[2]^2},\tilde{y}^{[2]^2},[\tilde{x}^{[2]},\tilde{y}^{[2]}]\in C_{n-3}(L)$. Then, by hypothesis we can write $[x,y^{[2]}]=\lambda_1x+\lambda_2x^{[2]}+\lambda_3y^{[2]}$ for some $\lambda_i\in \F$, $i=1,2,3$. Also, we have that $[[x,y^{[2]}],z]=0$ for any $z\in H$. For $z=y^{[2]}$ we obtain that $\lambda_1=0$, for $z=x$ we have $\lambda_3=0$ and, finally, for $z=y$ we get $[x^{[2]},y]=0$. Now, write $[x,y]=\lambda_4x+ \lambda_5x^{[2]}+ \lambda_6y+\lambda_7y^{[2]}$, for some $\lambda_i\in \F$, $i=4,\dots,7$. But then $[x,y^{[2]}]=[[x,y],y]=\lambda_4[x,y]$, and $0=[[x,y],y^{[2]}]=\lambda_4[x,y^{[2]}]$. Consequently, $\lambda_4=0$ and $[x,y^{[2]}]=0$, a contradiction.
\end{proof}


\begin{lemma}\label{l} Let $L$ be a restricted Lie algebra over an algebraically closed field of characteristic $p>0$ in which every restricted subalgebra is a restricted quasi-ideal.  If $H$ is a Cartan subalgebra of $L$, then $L$ has root space decomposition $$L=H\dot{+}(\oplus_{\alpha \in \Phi}(L_{\alpha}\dot{+}L_{-\alpha}) \oplus_{\beta \in \Psi} L_{\beta}),$$ where $\Phi$ is the set of roots $\alpha$ for which $-\alpha$ is also a root, and $\Psi$ is the remaining set of roots.
\end{lemma}
\begin{proof} Let $T$ be a maximal torus, $H=C_L(T)$ and let $L=H\dot{+}_{\alpha \in \Pi}L_{\alpha}$ be the corresponding root space decomposition. Then
\[ [x_{\alpha},x_{\beta}]=\lambda x_{\alpha}+\mu x_{\beta}+h \hbox{ for some } h \in H,
\] since $L_{\alpha}^{[p]}\subseteq H$ for all $\alpha \in \Pi$, by \cite[Corollary 4.3]{sf}. But $[L_{\alpha},L_{\beta}]\subseteq L_{\alpha+\beta}$, so, either $[L_{\alpha},L_{\beta}]=0$ or $[L_{\alpha},L_{\beta}]\subseteq H$ and $\alpha+\beta=0$. If $[L_{\alpha},L_{\beta}]=0$ for $\alpha\neq \beta$ then $[L_{-\alpha},L_{\beta}]=0$ also, giving the root space decomposition claimed.
\end{proof}
\medskip

Suppose that every restricted subalgebra of $L$ is a restricted quasi-ideal. Let $S$ be the subspace spanned by the semisimple elements of $L$ and let $P$ be the subspace spanned by the $p$-nilpotent elements of $L$. Then $S$ and $P$ are subalgebras of $L$, since $[x,y] \in \langle x \rangle_p + \langle y \rangle_p$, and, if $\F $ is perfect, $L=S+P$. Moreover, both are restricted, since
\[ (\lambda x + \mu y)^{[p]} = \lambda^px^{[p]} + \mu^py^{[p]} + \sum_{i=1}^{p-1} s_i(x,y),
\] and $x^{[p]}, y^{[p]}$ are semisimple (respectively, $p$-nilpotent) if so are $x, y$, and $s_i(x,y)\in \langle x,y \rangle^p$.

\begin{prop}\label{nilp} Let $L$ be a nilpotent restricted Lie algebra over a perfect field of characteristic different from $2$. Then every restricted subalgebra of $L$ is a restricted quasi-ideal of $L$ if and only if $L=S\oplus P$, where $S$ is a toral ideal and $P$ is a p-nilpotent ideal in which every restricted subalgebra is a restricted quasi-ideal. 
\end{prop}
\begin{proof} Suppose that every restricted subalgebra of $L$ is a restricted quasi-ideal of $L$ . By Proposition~\ref{p:car2}, $L^3=0$ and $L^{[p]}\subseteq Z(L)$. Then, for all $x,y\in L$, $(x+y)^{[p]}=x^{[p]}+y^{[p]}$, so $S, P$ are just the sets of semisimple and $p$-nilpotent elements of $L$ respectively. Then $S\cap P=0$ and $S\subseteq Z(L)$. It follows that $L=S\oplus P$ and that $S$ is toral. 
\par

The converse is straightforward.
\end{proof}

\begin{coro}\label{nilpc} Let $L$ be a restricted Lie algebra over an algebraically closed field of characteristic different from 2 in which every restricted subalgebra of $L$ is a restricted quasi-ideal of $L$. Then $L$ has a Cartan subalgebra $H$ such that $H=S\oplus P$ where $S$ is a torus and $P$ is the set of p-nilpotent elements in $H$, and $L=S\dot{+} N$ where $N$ is an ideal, $N^3=0$ and $N^{[p]}\subseteq Z(H)$. 
\end{coro}
\begin{proof} We have that $L$ has the form given in Lemma \ref{l} and $H=S\oplus P$, by Proposition \ref{nilp}. Now $L_{\alpha}^2=L_{-\alpha}^2=L_{\beta}^2=0$ since $2\alpha$, $-2\alpha$ and $2\beta$ are not roots. For every $h\in H$, $\alpha \in \Pi=\Phi \cup \Psi$, we have that $[h,x_{\alpha}]\in (\langle h\rangle_p+\langle x_{\alpha}\rangle_p)\cap L_{\alpha}$, so $[h,x_{\alpha}]=\lambda x_{\alpha}$ for some $\lambda \in \F $; that is, $h$ acts semisimply on $L_{\alpha}$. Also $\alpha(x_{\alpha}^{[p]})=0$, by \cite[Chapter 2, Corollary 4.3 (4)]{sf}. It follows that $[x_{\alpha}^{[p]}, x_{-\alpha}]=0$. Similarly, $[x_{-\alpha}^{[p]}, x_{\alpha}]=0$. Now $[x_{\alpha},x_{-\alpha}]\in \langle x_{\alpha}^{[p]}\rangle_p + \langle x_{-\alpha}^{[p]}\rangle_p$, so, if  $N=P+\sum_{\alpha\in\Phi}(L_{\alpha}+L_{-\alpha})+\sum_{\beta\in\Psi} L_{\beta}$  we have  $N^3=0$  and $N^{[p]}\subseteq Z(H)$. 
\end{proof}



\section{$J$-algebras and lower semimodular restricted Lie algebras}

 For this section, it will be useful to handle the following result. 

\begin{lemma}\label{flag} Let $L$ be a restricted Lie algebra over an algebraically closed field of characteristic $p>0$. If $L$ is supersolvable, then $L$ admits a complete flag made up of restricted ideals of $L$.
\end{lemma}
\begin{proof} Plainly, it is enough to show that $L$ has a one-dimensional restricted ideal, from which the conclusion will follow by induction. Suppose $\dim L>1$, the claim being trivial otherwise.  Consider a complete flag  
$$
0=L_0\subsetneq L_1 \subsetneq \cdots \subsetneq L_n=L
$$
of ideals of $L$.  If the ideal $L_1$ is restricted, then we are done. Thus we can suppose that there exists $x\in L_1$ such that $x^{[p]}\notin L_1$. As $L_1$ is an abelian ideal, the restricted subalgebra $H$ generated by  $x^{[p]}$ is contained in the centre of $L$.  Since the ground field is algebraically closed, by \cite[Chapter 2,  Theorem 3.6]{sf} we see that $H$ contains a toral element $t$. We conclude that $I=\F  t$ is a one-dimensional restricted ideal of $L$, as desired.
\end{proof}

Note that the assumption that the ground field is algebraically closed is essential for the validity of Lemma \ref{flag}. In fact, over arbitrary fields of positive characteristic, there can be cyclic restricted Lie algebras of arbitrary dimension with no non-zero proper restricted subalgebras (cf. \cite[Proposition 3.1]{mps}).  

\medskip
Let $L$ be a restricted Lie algebra. A restricted subalgebra $U$ of $L$ is called {\em lower semimodular} in $L$ if $U\cap B$ is maximal in $B$ for every restricted subalgebra $B$ of $L$ such that $U$ is maximal in $\langle U,B\rangle_p$. We say that $L$ is {\it lower semimodular} if every restricted subalgebra of $L$ is lower semimodular in $L$. 
\par

If $U$, $V$ are restricted subalgebras of $L$ with $U\subseteq V$, a {\em J-series} (or {\em Jordan-Dedekind series}) for $(U,V)$ is a series
\[ U=U_0\subsetneq U_1\subsetneq \ldots \subsetneq U_r=V
\] of restricted subalgebras such that $U_i$ is a maximal subalgebra of $U_{i+1}$ for $0\leq i \leq r-1$. This series has {\em length} equal to $r$. We shall call $L$ a {\em J-algebra} if, whenever $U$ and $V$ are restricted subalgebras of $L$ with $U\subseteq V$, all $J$-series for $(U,V)$ have the same finite length, $d(U,V)$. Put $d(L)=d(0,L)$. 

\begin{prop}\label{p:equiv} For a solvable restricted Lie algebra $L$ over an algebraically closed field of characteristic $p>0$ the following are equivalent:
\begin{itemize}
\item[(i)] $L$ is lower semimodular;
\item[(ii)] $L$ is a $J$-algebra; and
\item[(iii)] $L$ is supersolvable. 
\end{itemize}
\end{prop}
\begin{proof} (i)$\Rightarrow$(ii): This is just a lattice theoretic result (see \cite[Theorem V3]{birkhoff}).
\par

\noindent (ii)$\Rightarrow$(iii): We first show by induction on $\dim L$ that there exists a series of restricted subalgebras from $0$ to $L$ having length $\dim L$. Suppose $L\neq 0$. As $L$ is solvable, 
 it holds that $\langle L^{(1)} \rangle_p \neq L$; otherwise, $L^{(1)}=\langle L^{(1)} \rangle_p^{(1)}=L^{(2)}\neq 0$, a contradiction. Then the inductive hypothesis ensures the existence of a series of restricted subalgebras \chn
$$
U=U_0\subsetneq U_1\subsetneq \ldots \subsetneq U_r=\langle L^\prime \rangle_p
$$
with $\dim U_i=i$ for all $0\leq i \leq r$. Moreover, as $L/\langle L^\prime \rangle_p$ is abelian, Lemma \ref{flag} yields the claim. 

Now, by hypothesis, all $J$-series of restricted subalgebras from $0$ to $L$ have length $\dim L$, and consequently all maximal restricted subalgebras have codimension one in $L$. On the other hand, if $H$ is a maximal subalgebra of $L$ which is not restricted, then pick an element $x$ of $H$ such that $x^{[p]}\notin H$. Then $H+\F  x^{[p]}$ is a subalgebra of $L$ properly containing $H$, so $H+\F  x^{[p]}=L$ by the maximality of $H$.  Therefore, every maximal subalgebra has codimension one in $L$, which allows to conclude that $L$ is supersolvable, by \cite[Theorem 7]{barnes}.

(iii)$\Rightarrow$(i): Let $U,B$ be restricted subalgebras of $L$ such that $U$ is maximal in $\langle U,B\rangle_p$. By Lemma \ref{flag}, $U$ has codimension 1 in $\langle U,B\rangle_p$, which forces $\langle U,B\rangle_p =U+B$. It follows that $\dim(B/(U\cap B))=\dim((U+B)/U)=1$, whence $U\cap B$ is maximal in $B$, completing the proof.
\end{proof}

Note that the assumption of solvability is actually needed in the previous result. In fact, consider the restricted Lie algebra $L=\mathfrak{sl}(2,\F )$ over an algebraically closed field $\F $ of characteristic $p>2$. Then all $J$-series of restricted subalgebras of $L$ have length 3, despite the fact that $L$ is simple.

\section{Upper semimodular restricted Lie algebras}

Let $L$ be a restricted Lie algebra.
We say that a restricted subalgebra $S$ of $L$ is {\em upper semimodular} in $L$ if $S$ is maximal in $\langle S,T\rangle_p$ for every restricted subalgebra $T$ of $L$ such that $S\cap T$ is maximal in $T$. The restricted Lie algebra $L$ is called {\em  upper semimodular} if all of its restricted subalgebras are upper semimodular in $L$.

This section is devoted to study the structure of upper semimodular restricted Lie algebras over algebraically closed fields. In particular, our main aim of this section is to prove the following result:

\begin{theor}\label{main:upper}
	Let $L$ be a restricted Lie algebra over an algebraically closed  field. The following conditions are equivalent:
	\begin{enumerate}
\item[(i)] $L$ is upper semimodular;
\item[(ii)] $L$ is modular;
\item[(iii)] every restricted subalgebra of $L$ is a restricted quasi-ideal.
	\end{enumerate} Moreover, if one of the previous statements holds, then $L$ is either almost abelian or nilpotent of class at most $2$. 
\end{theor}


We start with some preliminary results.
%


Let $L$ be an almost abelian Lie algebra over a field $\F $ of characteristic $p>0$. Then  $L=\F x \dot{+} A$, where $A$ is an abelian ideal and $\ad(x)$ acts as the identity map on $A$. It is immediate to check that $L$ is restrictable and also centerless, so it admits a unique $p$-mapping by~\cite[Chapter 2, Corollary 2.2]{sf}. Explicitly, this $p$-mapping is given by $a^{[p]}=0$ for all $a\in A$ and $x^{[p]}=x$.


\begin{lemma}\label{prop:min}
	Let $L$ be an upper semimodular restricted Lie algebra over an algebraically closed field of characteristic $p>0$. If $L$ is generated by two distinct one-dimensional restricted subalgebras $X$ and $Y$, then $L$ is two-dimensional. 
	
%
	\end{lemma}

\begin{proof}
	Let $Z$ be a nonzero proper restricted subalgebra of $L$. Assume first that $X\subseteq Z, Y\not\subseteq  Z$. As $X\cap Y =0$ is maximal in $Y$, $X$ must be maximal in $L$, yielding $Z=X$. 
	Assume now that $X,Y\not\subseteq Z$ and take a one-dimensional restricted subalgebra $Z'$ of $Z$. By the previous case, 
	 $\langle X,Z'\rangle_p=L$. Since $X\cap Z'=0$ is maximal in $X$, $Z'$ is maximal in $L$ and $Z=Z'$. Thus, all nonzero proper restricted subalgebras of $L$ are one-dimensional,  and it follows from~\cite[Lemma 1.6]{Z} that $L$ is two-dimensional.   
\end{proof}






\begin{lemma}\label{lema:center}
	Let $\F $ be an algebraically closed field of characteristic $p>0$.
	Let $L$ be a non-abelian upper semimodular restricted Lie algebra over $\F $ generated by three one-dimensional restricted subalgebras. Then, $L$ is centerless.
\end{lemma}

\begin{proof}
	Let $\F x$, $\F y$, $\F z$ be three distinct one-dimensional restricted subalgebras generating $L$ and suppose, by contradiction, that $Z(L)\neq 0$. Note that 
	we can take $x$ to be either toral or such that $x^{[p]}=0$. By Lemma~\ref{prop:min} and without loss of generality, we may also assume $x\in Z(L)$ and that $\langle y,z\rangle_p$ is almost abelian, with $[y,z]=z$, $y^{[p]}=y$ and $z^{[p]}=0$. 
If $x^{[p]}=0$, then 
$\langle x+z\rangle_p\cap \F y=0$ is maximal in $\langle x+z\rangle_p$, but $\F y$ is not maximal in $\langle x+z,y\rangle_p=L$, a contradiction. 
On the other hand, if $x$ is toral, then
\[x\in \langle x-z\rangle_p\subseteq \langle x+y,y+z\rangle_p,\] so $\langle x+y,y+z\rangle_p=L$. Now $\langle x+y\rangle_p\cap\langle y+z\rangle_p =0$ is maximal in $\langle x+y\rangle_p$, but  $\langle y+z\rangle_p$ is not maximal in $L$, a contradiction. 
\end{proof}

\begin{prop}\label{prop:clas}
	Let $\F $ be an algebraically closed field of characteristic $p>0$. Any upper semimodular restricted Lie algebra $L$ over $\F $ generated by its one-dimensional restricted subalgebras is either abelian or almost abelian.
\end{prop}

\begin{proof}
	By Lemma~\ref{prop:min}, all the restricted subalgebras of $L$ generated by two one-dimensional restricted subalgebras are abelian or almost abelian. Suppose that $\langle y, x_1\rangle_p$ is almost abelian, where $\F y, \F x_1$ are restricted subalgebras of $L$ with $[y,x_1]=x_1$, $y^{[p]}=y$ and $x_1^{[p]}=0$. Write $L=\langle y, x_1,\dots,x_s\rangle_p$, where $y,x_1,\dots,x_s$ are linearly independent.
	We claim that $\langle y, x_i\rangle_p$ is almost abelian for $i=2,\dots,s$. Suppose otherwise that $[y,x_i]=0$ for some $i\neq 1$. By Lemma~\ref{lema:center}, we must have $[x_1,x_i]\neq 0$. Then $\langle x_1,x_i\rangle_p$ would be almost abelian 
	and 
	$[x_1,x_i]=\lambda x_1$ for some $\lambda\in \F $, $\lambda\neq 0$. But then $y+\lambda^{-1}x_i\in Z(\langle y,x_1,x_i\rangle_p)=0$ by Lemma~\ref{lema:center}, a contradiction. 
	Note also that $[y,x_i]\notin \F  y$, as otherwise $y^{[p]}=0$. Therefore,  we can clearly 
	assume that $[y,x_i]=x_i$.  
	For $i\neq j$ write $[x_i,x_j]=\alpha_{ij}x_i+\beta_{ij}x_j$. 
	We have
	\begin{align*}
	0&=[[y,x_i],x_j]+[[x_i,x_j],y]+[[x_j,y],x_i]\\
	&=\alpha_{ij}x_i+\beta_{ij}x_j - \alpha_{ij}x_i-\beta_{ij}x_j +\alpha_{ij}x_i+\beta_{ij}x_j\\
	&=\alpha_{ij}x_i+\beta_{ij}x_j,
	\end{align*}
	hence $\alpha_{ij}=\beta_{ij}=0$.
	
	
	Therefore, $L=\langle x_1,\dots,x_s\rangle_p\dot{+} \F y$ is an almost abelian restricted Lie algebra of dimension $s+1$, as desired.
\end{proof}

Note that the hypothesis of $\F $ being algebraically closed is essential for our results. Indeed, the Lie algebra $L$ over a perfect field of characteristic $3$ given by Gein in~\cite[Example 2]{gein},  with the $p$-mapping indicated in Section~\ref{s:da}, is upper semimodular, generated by its minimal restricted subalgebras and semisimple. 
The reader could ask if, ruling out the hypothesis of $\F $ being algebraically closed, any  upper semimodular restricted Lie algebra generated by its minimal restricted subalgebras would be abelian, almost abelian or semisimple, in a way somehow similar to the situation in the ordinary Lie algebra setting (see~\cite{gein2}). 
However, this is not the case either: the restricted Lie algebra $\F x \oplus L$, with $x^{[p]}=0$, is generated by its minimal restricted subalgebras and it is  upper semimodular, but it is neither abelian, nor almost abelian, nor semisimple. 
Furthermore, it is even possible to pick a  modular restricted subalgebra of  $\F x \oplus L$ which does not lie in any of these three cases.

\begin{prop}\label{prop:usmalab}
	Let $\F $ be an algebraically closed field of characteristic $p>0$, and let $L$ be an upper semimodular restricted Lie algebra over $\F $. Let $B$ be the restricted subalgebra generated by the one-dimensional restricted subalgebras of $L$. If $B$ is almost abelian, then $L=B$.
\end{prop}

\begin{proof}
	Assume  $L\neq B$. By Proposition~\ref{prop:atom}, there exists a $p$-nilpotent element $x\in L$ of order $2$. Write $B=A\dot{+}\F y$, where $A$ is a strongly abelian restricted ideal of $B$, and $y$ is a toral element which acts as the identity map on $A$. Since $x^{[p]}\in A$, we have $\ad^p(x)(y)=[x^{[p]},y]=-x^{[p]}$. Set $w=\ad^{p-1}(x)(y)$, and note that $[x,w]=-x^{[p]}$ and $[x^{[p]},w]=[x,w^{[p]}]=0$.
	
	As $\langle x \rangle_p \cap \langle x^{[p]},y \rangle_p = \F x^{[p]}$ is maximal in $\langle x^{[p]},y \rangle_p=\F x^{[p]}+\F y$, one has that $\langle x \rangle_p$ must be maximal in $\langle x,x^{[p]},y \rangle_p= \langle x,y \rangle_p$. We have 
	
	\[\langle x \rangle_p\subsetneq \langle x,w \rangle_p \subseteq \langle x,y \rangle_p.\] 
	It follows that $y\in \langle x,w \rangle_p= \langle x \rangle_p + \langle w \rangle_p$, from which $[x,y]=\lambda[x,w]=-\lambda x^{[p]}$, for some $\lambda\in \F $. But then \[-x^{[p]}=\ad^{p}(x)(y)=-\lambda \ad^{p-1}(x)(x^{[p]})=0,\] a contradiction. Therefore, $L=B$ and $L$ is almost abelian. 
\end{proof}

\begin{theor}\label{th:uppergral}
	Let $\F $ be an algebraically closed field of characteristic $p>0$. Any  upper semimodular restricted Lie algebra $L$ over $\F $ is either abelian, almost abelian or of the form
	\[L=\langle x_1,\dots,x_r,B\rangle_p,\]
	where $x_i$ is $p$-nilpotent of order $n_i>1$ for all $i=1,\dots,r$, $B$ is an abelian restricted subalgebra and $[L,L]\subseteq \langle x_1,\dots,x_r\rangle_p$.
\end{theor}

\begin{proof}
	Let $B$ be the restricted subalgebra generated by the one-dimensional restricted subalgebras of $L$. By Proposition~\ref{prop:clas}, $B$ is either abelian or almost abelian. If $L\neq B$, then $B$ is abelian by Proposition~\ref{prop:usmalab}, and every $x_i\notin B$ is $p$-nilpotent of order $n_i>1$ by Proposition~\ref{prop:atom}. 
	
	To prove that $[L,L]\subseteq \langle x_1,\dots,x_r\rangle_p$, it suffices to see that $[x_i,b]\in \langle x_i\rangle_p$, for $i=1,\dots,r$ and $b\in B$ such that $\langle b\rangle_p$ 
	is one-dimensional.
	Take such a $b\in B$. If $b\in \langle x_i \rangle_p$, then we are done. Otherwise,  $\langle x_i \rangle_p \cap \langle b \rangle_p=0$ is maximal in $ \langle b \rangle_p=\F b$, and then $ \langle x_i \rangle_p$ must be maximal in  $\langle x_i,b \rangle_p$.
	Write $w=\ad^{r-1}(x_i)(b)\neq 0$, where $r$ is such that $\ad^r(x_i)(b)=0$. We have the following chain of inclusions 
	\[ \langle x_i \rangle_p \subseteq  \langle x_i,w \rangle_p\subsetneq  \langle x_i,b \rangle_p.\]
	Then, $w\in  \langle x_i \rangle_p$. 
	Assume now that $\ad^{r-k}(x_i)(b)\in \langle x_i \rangle_p$ for some $k>1$, and set $w'=\ad^{r-k-1}(x_i)(b)$. Again, it is clear that
	\[ \langle x_i \rangle_p \subseteq  \langle x_i,w' \rangle_p\subseteq  \langle x_i,b \rangle_p,\] where one inclusion has to be an equality.
	By assumption,  if $b\in  \langle x_i,w' \rangle_p= \langle x_i \rangle_p +  \langle w' \rangle_p$, then $[x_i,b]\in  \langle x_i \rangle_p$. Therefore $w'\in  \langle x_i \rangle_p$, and by induction  we have that $[x_i,b]\in  \langle x_i \rangle_p$.
\end{proof}

	Note that, although any abelian or almost abelian restricted Lie algebra is upper semimodular, the converse of Theorem~\ref{th:uppergral} does not hold, as the following example shows.

\begin{ex}\label{ex:usmn}\emph{
 Let $L=\langle x,y,z\rangle_p$ with $x^{[p]^2}=y^{[p]}=z^{[p]}=0$ and $[x,y]=z$ as the only non-zero product. Then the restricted subalgebra  $B=\F x^{[p]}\oplus \F y \oplus \F z$ generated by all the one-dimensional restricted subalgebras is abelian. However, $L$ is not upper semimodular, as $\langle x\rangle_p \cap \F y =0$ is maximal in $\F y$, but $\langle x\rangle_p$ is not maximal in $\langle x,y\rangle_p=L$.}
\end{ex}



\begin{prop}\label{p:nilp}
	Let $\F $ be an algebraically closed field of characteristic $p>0$, and let $L$ be an upper semimodular restricted Lie algebra over $\F $. Then, $L$ is  
	almost abelian or
	nilpotent. 
\end{prop}

\begin{proof} Assume that $L$ is not almost abelian.
Let $T$ be a torus of $L$. By~\cite[Chapter~2, Theorem 3.6]{sf}, $T$ has a basis consisting of toral elements and therefore $T\subseteq B$, in the notation of Theorem~\ref{th:uppergral}. 
Consequently, every semisimple element of $L$ belongs to $B$, and the restricted subalgebra $\mathfrak{T}$ formed by the semisimple elements of $L$ is the unique maximal torus of $L$. Suppose, by contradiction, that $L$ is not nilpotent. Consider the Cartan subalgebra $H=C_L(\mathfrak{T})$ and the associated root space decomposition $L=H\dot{+}(\sum_{\alpha\in \Phi}L_{\alpha})$.  Then there exists $\alpha\in \Phi$ and a toral element $t\in \mathfrak{T}$ such that $\alpha(t)\neq 0$. Let $x\in L_{\alpha}$, $x\neq 0$. By  
\cite[Chapter 2, Corollary 4.3(1)]{sf}, we have $[t,x]=\alpha(t)x$ and $\alpha(t) \in \mathrm{GF}(p)$. Thus one has
$$
(t+x)^{[p]}=t+x^{[p]}+\alpha(t)^{p-1}x=t+x^{[p]}+x.
$$
Moreover, by  \cite[Chapter 2, Corollary 4.3(3)]{sf} we have that $x^{[p]}\in H$ and so $[t,x^{[p]}]=0$. By induction, it follows that
\begin{equation}\label{t+x}
(t+x)^{[p]^n}=t+ \sum_{i=0}^n x^{[p]^{n}}
\end{equation}
for every $n>0$. Now, by \cite[Chapter 2, Theorem 3.4]{sf} we see that $(t+x)^{[p]^n}	\in \mathfrak{T}$ for some sufficiently large $n$, and so we deduce from (\ref{t+x}) that $x\in H$, a contradiction.
\end{proof}

\begin{coro}\label{c:lj}
	Let $\F $ be an algebraically closed field of characteristic $p>0$, and let $L$ be an upper semimodular restricted Lie algebra over $\F $. Then, $L$ is also lower semimodular and a $J$-algebra.
\end{coro}
\begin{proof}
	It follows from Proposition~\ref{p:nilp} and Proposition~\ref{p:equiv}.
\end{proof}

\begin{prop}\label{p:qi}
Let $\F $ be an algebraically closed field of characteristic $p>0$, and let $L$ be an upper semimodular restricted Lie algebra over $\F $. Then, every restricted subalgebra of $L$ is a restricted quasi-ideal.
\end{prop}

\begin{proof}
By Proposition~\ref{p:nilp}, $L$ is either almost abelian or nilpotent. If $L$ is almost abelian, then we are done, so suppose that it is nilpotent. Let $x,y\in L$. If $x,y$ are semisimple, then we have that $x,y\in B$ and $[x,y]=0$. If $x$ is semisimple and $y$ is $p$-nilpotent, then $x\in B$ and we get that $[x,y]\in\langle y\rangle_p$ as in the proof of Theorem~\ref{th:uppergral}. If $x,y$ are $p$-nilpotent, we claim that $[x,y]\in\langle x\rangle_p+\langle y\rangle_p$. Indeed, let $s$ be the sum of their orders of $p$-nilpotency. We will proceed by induction on $s$. If $s=2$, then $x,y\in B$ and therefore $\langle x,y \rangle_p \subseteq  \langle x \rangle_p + \langle y \rangle_p$. Fix now $s>2$, and assume that $x^{[p]}\neq 0$. If $x\in \langle x^{[p]},y \rangle_p$, it holds that $\langle x,y \rangle_p =\langle x^{[p]},y \rangle_p $ is contained in $\langle x^{[p]} \rangle_p + \langle y \rangle_p$ by induction. Otherwise,  $\langle x^{[p]} \rangle_p = \langle x \rangle_p \cap \langle x^{[p]},y \rangle_p$ is maximal in $\langle x \rangle_p$, so $\langle x^{[p]},y \rangle_p$ is maximal in $ \langle x,y \rangle_p$. Then $\langle x^{[p]},y \rangle_p$ has codimension one in $ \langle x,y \rangle_p$ and $\langle x,y \rangle_p = \langle x \rangle_p + \langle x^{[p]},y \rangle_p$. But by induction, $\langle x^{[p]},y \rangle_p\subseteq \langle x^{[p]} \rangle_p + \langle y \rangle_p$.

Now take $x,y$ two arbitrary elements in $L$ and consider their Jordan-Chevalley decompositions, $x=x_s + x_n$ and $y=y_s + y_n$. The above arguments show that $[x,y]\in\langle x_n\rangle_p + \langle y_n\rangle_p$. Since $x_s^{[p]^r}\in \langle x \rangle_p $ and $y_s^{[p]^t}\in \langle y \rangle_p $ for $r$ and $t$ large enough and $x_s,y_s$ are semisimple, we get that $x_n\in \langle x \rangle_p $ and $y_n\in \langle y \rangle_p $. It follows that $[x,y]\in\langle x\rangle_p + \langle y\rangle_p$.
\end{proof}

The following simple lemma is all what is left to prove Theorem~\ref{main:upper}. We need an easy consideration first.

Let $X$ be a restricted quasi-ideal of a restricted Lie algebra $L$. Then, for every restricted subalgebra $Y$ of $L$, it holds that $X+Y=\langle X,Y\rangle_p$ is a restricted subalgebra of $L$.

\begin{lemma}\label{lemma:mod}
	Let $L$ be a restricted Lie algebra in which every restricted subalgebra is a restricted quasi-ideal. Then, $L$ is  modular, and consequently, upper semimodular and  lower semimodular.
\end{lemma}

\begin{proof}
	Let $X$, $Y$ and $Z$ be restricted subalgebras of $L$ such that $X\subseteq Z$. Take $z\in\langle X,Y\rangle_p\cap Z=(X+Y)\cap Z$, and write $z=x+y$ for some $x\in X$, $y\in Y$. Then $x\in Z$, yielding that $y\in Y\cap Z$. Therefore, $z\in X + (Y\cap Z)=\langle X, Y\cap Z \rangle _p$. Then, 
	$L$ is modular.
\end{proof}

It is now a simple matter to prove the main result of this section:

\smallskip
\emph{Proof of Theorem \ref{main:upper}.} It follows from the combination of Proposition~\ref{p:nilp}, Proposition~\ref{p:qi}, Proposition~\ref{p:car2} and Lemma~\ref{lemma:mod}. \qed 

\section*{Acknowledgements}
We are deeply grateful to A. Premet for some useful discussions.

The first author was supported by Agencia Estatal de Investigaci\'on (Spain), Grant PID2020-115155GB-I00 (European FEDER support included, UE), and by Xunta de Galicia, Grant ED431C 2019/10 (European FEDER support included, UE).


\end{document}